\documentclass[a4paper,12pt]{amsart}

\usepackage[cp1251]{inputenc}
\usepackage[english, ukrainian]{babel}
\usepackage[all,dvipdfm]{xy}
\usepackage{amssymb,upref}
\usepackage[dvips]{graphicx}
\graphicspath{{C:\Users\Aspirant\Desktop\Новая папка}}
\usepackage{setspace}
\usepackage{xcolor}
\usepackage{fancyhdr}
\usepackage{graphicx}

\theoremstyle{plain}
\newtheorem{theorem}{Theorem}
\newtheorem{lemma}{Lemma}
\newtheorem{corollary}{Corollary}

\newtheorem{prop}{Proposition}

\theoremstyle{definition}

\theoremstyle{remark}
\newtheorem{remark}{Remark}
\newtheorem{proper}{Property}

\newtheorem{statment}{Statment}
\newtheorem{defin} {Definition}
\newtheorem*{main_theorem}{Main Theorem}

\uccode`ґ=`Ґ \uccode`Ґ=`Ґ \lccode`Ґ=`ґ \lccode`ґ=`ґ %
\uccode`є=`Є \uccode`Є=`Є \lccode`Є=`є \lccode`є=`є %
\uccode`і=`І \uccode`І=`І \lccode`І=`і \lccode`і=`і %
\uccode`ї=`Ї \uccode`Ї=`Ї \lccode`Ї=`ї \lccode`ї=`ї %

\begin{document}
\selectlanguage{english}

\title[] {Minimal generating systems and properties of Sylow 2-subgroups of alternating group}

\author[]{Ruslan Skuratovskii}


\begin{abstract}

   The background of this paper is the following:  search  of the minimal systems of generators for this class of group which still was not founded also problem of representation for this class of group, exploration of systems of generators for Sylow 2-subgroups $Syl_2{A_{2^k}}$ and $Syl_2{A_{n}}$ of alternating group, finding  structure these subgroups.

The authors of \cite{Dm} didn't proof minimality of finding by them system of generators for such Sylow 2-subgroups of $A_n$ and structure of it were founded only descriptively.
	The aim of this paper is to research the structure of Sylow 2-subgroups and to construct a minimal generating system for such subgroups. In other words, the problem is not simply in the proof of existence of a generating set with elements for Sylow 2-subgroup of alternating group of degree $2^k $ and its constructive proof and proof its minimality. For the construction of minimal generating set we used the representation of elements of group by automorphisms of portraits for binary tree. Also, the goal of this paper is to investigate the structure of 2-sylow subgroup of alternating group more exactly and deep than in \cite{Dm}.
 The main result is the proof of minimality of this generating set of the above described subgroups and also the description of their structure.

\end{abstract}

\maketitle

\section{Introduction }
The aim of this paper is to research the structure of Sylow 2-subgroups of $A_{2^k}$, $A_n$ and to construct a minimal generating system for it. Case of Sylow subgroup where $p=2$ is very special because group $C_2\wr C_2\wr C_2 \ldots \wr C_2 $ admits odd permutations, this case was not fully investigated in \cite{Dm,Sk}. There was a mistake in a statement about irreducibility that system of $k+1$ elements for $Syl_2(A_{2^k})$ which was in abstract \cite{Iv} on Four ukraine conference
of young scientists in 2014 year.
All undeclared terms are from \cite{Ne, Gr}.
A minimal system of generators for a Sylow 2-subgroups of $A_n$ was found.
 We denote by $v_{j,i}$ the vertex of $X^j$, which has the number $i$.

\section{ Main result  }

Let's denote by $X^{[k]}$ a regular truncated binary rooted tree (with number of levels from $0$ to $k$ but active states are only from $X^l, 0 \leq l \leq k-1$) labeled by vertex, $v_{j,i} X^{[k-j]} $ -- subtree of $X^{[k]}$ with root in $v_{j,i}$.
If the state in the vertex is non-trivial, then its label is 1, in other case
it is 0.
An automorphism of $X^{[k]}$ with non-trivial state in some of
 $v_{1,1}$, $v_{1,{2}}$, $v_{2,{1}}$,..., $v_{2,4}$, ... ,$v_{m,1}$, ... ,$v_{m,j}$, $ m < k, j \leq 2^m$ is denoted by $\mathop{\beta}_{1,(i_{11},i_{12});...; l,(i_{l1},...,i_{l2^l});...;{m},(i_{m1},...,i_{m2^m})}$ 
where the indexes $l$ is the numbers of levels with non-trivial states, in
parentheses after this numbers we denote a cortege of vertices of this level,  $i_{mj}=0$ if state in $v_{mj}$ is trivial $i_{mj}=1$ in other case. If for some $l$ all $i_{lj}=0$ then cortege $l,(i_{l1} ,..., i_{l 2^l})$ does not figure in indexes of $\beta$. But if numbers of vertexes with active states are certain, for example $v_{j,1}$ and $v_{j,s}$ we can use more easy notation $\mathop{\beta}_{j,(1,s);}$. If in parentheses only one index then parentheses can be omitted for instance $\mathop{\beta}_{j,(s);}=\mathop{\beta}_{j,s;}$.
Denote by $\tau_{i,...,j}$ the automorphism, which has a non-trivial vertex permutation (v.p.)
only in vertices $v_{k-1,i}$, ... ,$v_{k-1, j}$, $j  \leq 2^{k-1} $ of the
level $X^{k-1}$.
Denote by $\tau$ the automorphism $\tau_{1,2^{k-1}}$.
Let's consider special elements such that: $ \mathop{\alpha}_{0}=\mathop{\beta}_{0}=\mathop{\beta}_{0,(1,0,...,0)}, \mathop{\alpha}_{1}=\mathop{\beta}_{1}=\mathop{\beta}_{1,(1,0,...,0)}
, \ldots ,\mathop{\alpha}_{l}=\mathop{\beta}_{l}=\mathop{\beta}_{l,(1,0,...,0)} $.

 \begin{lemma} \label{even} 

  Every state from $X^l$, $l<k-1$ determines an even permutation on $X^{k}$.

\end{lemma}
\begin{proof}
Actually every transposition from $X^l$, $l<k-1$ acts at even number of pair of vertexes because of binary tree structure, so it realize even permutation (of set which consists of some vertexes of $X^{k}$) with cyclic structure \cite{Sh}  $(1^{2^{k-1}-2^{k-l-l}},2^{2^{k-l-l}})$ because it formed by the structure of binary tree.
 \end{proof}

\begin{corollary}\label{B_k-1}

Due to Lemma \ref{even} vertex permutations from $Aut X^{[k-1]}=\langle \mathop{\alpha}_{0}. ... , \mathop{\alpha}_{k-2} \rangle$ form a group: ${B_{k-1}}= \underbrace {C_2 \wr ...\wr C_2}_{k-1} $ which acts at $X^{k-1}$ by even permutations. Order of $B_{k-1}$ equal to $2^{2^{k-1}-1}$.

\end{corollary}

Let us denote by $G_k$ the subgroup of $Aut X^{[k]}$ such that $G_k  \simeq  Syl_2  A_{2^k}$
and $W_{k-1}$ the subgroup of $G_k$ that has active states only on $ X^{k-1}$.

\begin{prop} \label{ordW}

An order of ${{W}_{k-1}}$ is equal to ${{2}^{{{2}^{k-1}}-1}},\,\,k > 1$ its structure is $W_{k-1} \simeq (C_2)^{{{2}^{k-1}}-1}$.
\end{prop}

  \begin{proof}
On ${{X}^{k-1}}$ we have ${{2}^{k-1}}$ vertexes where can be group ${{V}_{k-1}}\simeq {{C}_{2}}\times {{C}_{2}}\times ...\times {{C}_{2}}\simeq {{({{C}_{2}})}^{k-1}}$, but as a result of the fact that ${X}^{k-1}$ contains only even number of active states or active vertex permutations (v.p.) in vertexes of $X^{k-1}$, there are only half of all permutations from ${{V}_{k-1}}$ on $X^{k-1}$. So it's subgroup of ${{V}_{k-1}}$: ${{W}_{k-1}}\simeq {}^{C_{2}^{{{2}^{k-1}}}}/{}_{{{C}_{2}}}$. So we can state: $|{{W}_{k-1}}|={2^{k-1}-1}$, $W_{k-1}$ has $k-1$ generators and we can consider ${W}_{k-1}$ as vector space of dimension $k-1$.
 \end{proof}

For example let's consider subgroup $W_{4-1}$ of $A_{2^4}$ its order is $2^{2^{4-1}-1}=2^7$ and $|A_{2^4}|=2^{14}$.

 \begin{lemma}\label{gen} 
 The elements $\tau $ and ${{\alpha }_{0}},...,{{\alpha }_{k-1}}$ generates arbitrary element ${{\tau }_{ij}}$.
\end{lemma}
 \begin{proof} According to [5, 6] 
  the set  ${{\alpha }_{0}},...,{{\alpha }_{k-2}}$ is minimal system of generators of group $Aut{{X}^{[k-1]}}$. Firstly, we shall proof the possibility of generating an arbitrary $\tau_{ij}$, from set $v_{(k-1,i)}$, $1\leq i \leq 2^{k-2} $. 
Since $Aut v_{1,1}X^{[k-2]} \simeq \left\langle {{\alpha }_{1}},...,{{\alpha }_{k-2}} \right\rangle $ acts at ${{X}^{k-1}}$ transitively it follows an ability to permute vertex with a transposition from automorphism $\tau $ and stands in $v_{k-1,1}$ in arbitrary vertex ${{v}_{k-1,j}},\,\,\,j\le {{2}^{k-2}}$ of $v_{1,1}X^{[k-1]}$, for this goal we act by $\alpha_{k-j}$ at $\tau $: $\alpha_{k-j} \tau \alpha_{k-j} ={{\tau }_{j, 2^{k-2}}}$. Similarly we act at $\tau$ by corespondent $\alpha_{k-i}$ to get $\tau_{i, 2^{k-2}}$ from $\tau $: $\alpha_{k-i}  \tau \alpha_{k-i}^{-1}={{\tau }_{i, 2^{k-2}}}$. Note that automorphisms $\alpha_{k-j}$ and $\alpha_{k-i}, 1<i,j<k-1$ acts only at subtree $v_{1,1}X^{[k-1]}$ that's why they fix v.p. in $v_{k-1, 2^{k-1}}$. Now we see that ${\tau }_{i, 2^{k-2}}{{\tau }_{j, 2^{k-2}}}={{\tau }_{i, j}}$, where $1 \leq i,j < 2^{k-2}$.
To get ${\tau }_{m, l}$ from $v_{1,2}X^{[k-1]}$, i.e. $2^{k-2} < m,l \leq 2^{k-1} $ we use $\alpha_0$ to map ${\tau }_{i, j}$ in ${\tau }_{i+2^{k-2}, j+2^{k-2}}\in v_{1,2} AutX^{[k-1]}$. To construct arbitrary transposition ${\tau }_{i,m}$ from $W_k$ we have to multiply ${\tau }_{1,i} {\tau } {\tau }_{m,2^{k-1}}={\tau }_{i,m}$.
Lets realize natural number of $v_{k,l}$, $1<l<2^k$ in 2-adic system of presentation (binary arithmetic).  Then $l={\delta_{{{l}_{1}}}}{{2}^{m_l}}+{\delta_{{{l}_{2}}}}{{2}^{m_l-1}}+...+{\delta_{{{l}_{m_l+1}}}},\,\, \delta_{l_i} \in \{0,1\}$ where is a correspondence between  ${\delta_{{{l}_{i}}}}$ that from such presentation and expressing of automorphisms: $\tau_{l,2^{k-1}} = \prod_{i=1}^{m_l} \alpha_{k-2-(m_{l}-i)}^{\delta_{{l}_{i}}} \tau  \prod_{i=1}^{m_l} \alpha_{k-2-(m_{l}-i)}^{\delta_{l_i}},  1 \leq m_l \leq k-2$.  In other words $\left\langle {{\alpha }_{0}},...,{{\alpha }_{k-2}},\tau  \right\rangle \simeq {{G}_{k}}$.
 \end{proof}

\begin{corollary}\label{genG_k-1} The elements from conditions of Lemma \ref{gen} is enough to generate basis of $W_{k-1}$.
\end{corollary}

 \begin{theorem} A maximal 2-subgroup $G_k$ in $Aut{{X}^{\left[ k \right]}}$, that acts by even permutation on ${{X}^{k}}$ has structure of inner semidirect product $B_{k-1} \ltimes {{W}_{k-1}}$.
\end{theorem}
\begin{proof}
A maximal 2-subgroup of  $Aut{{X}^{\left[ k-1 \right]}}$ is isomorphic to ${{B}_{k-1}}\simeq \underbrace{{{S}_{2}}\wr {{S}_{2}}\wr ...\wr {{S}_{2}}}_{k-1}$ (this group acts on ${{X}^{k-1}}$). A maximal 2-subgroup which has elements with active states only on ${{X}^{k-1}}$ corresponds subgroup ${W}_{k-1}$.
Since subgroups ${B}_{k-1}$ and ${{W}_{k-1}}$ are embedded in $Aut{{X}^{\left[ k \right]}}$, then define an action of ${{B}_{k-1}}$ on elements of ${{W}_{k-1}}$ as ${{\tau }^{\sigma }}=\sigma \tau {{\sigma }^{-1}},\,\,\,\sigma \in {{B}_{k-1}},\,\,\tau \in {{W}_{k-1}}$,
i.e. action of inner automorphism (inner action) from $Aut{{X}^{\left[ k \right]}}$.
Note that ${{W}_{k-1}}$ is subgroup of stabilizer of  ${{X}^{k-1}}$ i.e. ${{W}_{k-1}}<St_{Aut{X}^{[k]}}(k-1)\lhd AutX^{[k]}$ and is normal too $W_{k-1}\lhd AutX^{[k]}$, because conjugation keeps a cyclic structure of permutation so even permutation maps in even. Therefore such conjugation induce automorphism of ${W}_{k-1}$ and $G_k \simeq B_{k-1}\ltimes W_{k-1}$.

 Let us show it in other way with using another technique of permutations, for this goal to every $\sigma \in B_{k-1}$ ($\sigma(j)=i_j$, $1 \leq j \leq 2^{k-1}$), which acts on ${X}^{k-1}$ (${X}^{k-1}$ has $2^{k-1}$ elements) set in correspondence a permutation
$$
\beta_\sigma=\left(\begin{array}{rrrrrr}
  1, & 2, & 3, & 4, & ..., & 2^{k} \\
  2{{i}_{1}}-1, & 2{{i}_{1}}, & 2{i}_{2}-1, & 2{i}_{2}, & ..., & 2{i}_{2^{k-1}}
\end{array}  \right), i_j\in{1,...,2^{k-1}}.
$$
  Note that
$\beta_{\sigma }$ acts on ${{2}^{k}}$ elements as elements of ${{W}_{k-1}}$ so we can multiply  it.  I.e. for element
${{\beta }_{\sigma }}(2i-1)=2\sigma (i)-1,\,\,{{\beta }_{\sigma }}(2i)=2\sigma (i)$, ${{\sigma }_{i}}\in \left\{ {{1,2,...,2}^{k-1}} \right\}$.
Let us prove, that it is action of group really element $e$, $e\in{{B}_{k-1}}$ acts trivial.
This operation is associative, since ${{\tau }^{\sigma \phi }}={{({{h}^{\sigma }})}^{\phi }}=\phi \sigma \tau {{\sigma }^{-1}}{{\phi }^{-1}},\,\,\sigma ,\phi \in {{B}_{k-1}}$.

 Let us  check that it is homomorphism. Let us consider mapping ${{({{\tau }_{2k-1,2k}})}^{\sigma }}\mapsto {{\tau }_{2n-1,2n}}$, where ${{\tau }_{2n-1,2n}}$  is even transformation, so we completeness set of elements from ${{W}_{k-1}}$ under last mapping. Since it is homomorphous mapping in a group $Aut{{W}_{k-1}}$ from ${{B}_{k-1}}$.  It's semidirect product ${{B}_{k-1}}\ltimes {{W}_{k-1}}\simeq ( \underbrace{{{S}_{2}}\wr {{S}_{2}}\wr ...\wr {{S}_{2}}}_{k-1}) \ltimes (\underbrace{{{C}_{2}}\times {{C}_{2}}\times ...\times {{C}_{2}}}_{{{2}^{k-1}}-1}) $.
Order of ${{W}_{k-1}}$ is equal ${{2}^{{{2}^{k-1}}-1}}$ because at ${{X}^{k-1}}$ is ${{2}^{{{2}^{k-1}}}}$ vertexes but only half  of combinations of active states from  ${{X}^{k-1}}$ can form even permutation thus $\left| {{W}_{k-1}} \right|={{2}^{{{2}^{k-1}}-1}}$.
Since  at ${{X}^{k-1}}$ is ${{2}^{{{2}^{k-1}}}}$ vertexes and half of combinations of active states from  ${X}^{k-1}$ can form even permutation thus $\left| {{W}_{k-1}} \right|={{2}^{{{2}^{k-1}}-1}}$ that is proved in Proposition \ref{ordW}. Using the Corollary \ref{B_k-1} about ${{B}_{k-1}}$ order of  ${{G}_{k}}\simeq {{B}_{k-1}} \ltimes {{W}_{k-1}}$ is ${{2}^{{{2}^{k-1}}-1}}\cdot {{2}^{{{2}^{k-1}}-1}}={{2}^{{{2}^{k}}-2}}$.
\end{proof}


 \begin{lemma}\label{ordG_k} Orders of groups $G_k = \langle   \mathop{\alpha}_{0}, \mathop{\alpha}_{1},
 \mathop{\alpha}_{2},...,\mathop{\alpha}_{k-2}, \mathop{\tau} \rangle $
and $Syl_2(A_{2^{k}})$ are equal to $2^{2^{k}-2}$.
\end{lemma}

In accordance with Legender's formula, the power of 2 in ${{2}^{k}}!$ is $\left[ \frac{{{2}^{k}}}{2} \right]+\left[ \frac{{{2}^{k}}}{{{2}^{2}}} \right]+\left[ \frac{{{2}^{k}}}{{{2}^{3}}} \right]+...+\left[ \frac{{{2}^{k}}}{{{2}^{k}}} \right]=\frac{{{2}^{k}}-1}{2-1}$. We need to subtract 1 from it because we have only $\frac{n!} {2}$ of all permutations as a result: $\frac{{{2}^{k}}-1}{2-1}-1=2^{k}-2$. So $\left| Syl({{A}_{{{2}^{k}}}}) \right|={{2}^{{{2}^{k}}-2}}$.
The same order has group $G_k=B_{k-1} \ltimes W_{k-1}$: $|G_k|=|B_{k-1}|\cdot|W_{k-1}|= |Syl_2 A_{2^k}|$, since order of groups $G_{k}$ according to Proposition \ref{ordW} and the fact that $|B_{k-1}|=2^{2^{k-1}-1}$ is $2^{2^{k}-2}$.
For instance the orders of $Syl_2 (A_8)$, $B_{3-1}$ and $W_{3-1}$: $|W_{3-1}|= 2^{2^{3-1}-1}=2^3=8$, $|B_{3-1}|=|C_2\wr C_2| = 2 \cdot 2^2=2^3$ and according to Legendre's formula, the power of 2 in ${{2}^{k}}!$ is $\frac{{2}^{3}}{2}+ \frac{{2}^{3}}{2^2}+\frac{{2}^{3}}{2^3} -1=6$ so $Syl_2 (A_8) = 2^6=2^{2^k-2}$, where $k=3$. Next example for $A_{16}$: $Syl_2 (A_{16}) =2^{2^4-2}= 2^{14}, k=4$, $|W_{4-1}|= 2^{2^{4-1}-1}=2^7$, $|B_{4-1}|=|C_2\wr C_2\wr C_2| = 2 \cdot 2^2\cdot 2^4 = 2^7$. So we have equality $2^7 2^7 = 2^{14}$ which endorse the formula.

 \begin{theorem} \label{isomor}  
  The set $S_{\mathop{\alpha}}= \{\mathop{\alpha}_{0}, \mathop{\alpha}_{1},
 \mathop{\alpha}_{2},  ... ,\mathop{\alpha}_{k-2}, \mathop{\tau}\}$
   of elements from subgroup of $AutX^{[k]}$
   generates a group $G_k$ which isomorphic to $Syl_2(A_{2^{k}})$.
\end{theorem}

 \begin{proof}
As we see from Corollary \ref{B_k-1}, Lemma \ref{gen} and Corollary \ref{genG_k-1} group $G_k$ are generated by $S_{\mathop{\alpha}}$ and their orders according to Lemma \ref{ordG_k} is equal. So according to Sylow's theorems 2-subgroup $G_k<A_{2^k}$ is $Syl_2 (A_{2^k})$.
 \end{proof}
 Consequently, we construct a generating system, which contains $k$ elements, that is less than in \cite{Iv}.

The structure of Sylow 2-subgroup of $A_{2^k}$ is the following: $\underset{i=1}{\overset{k-1}{\mathop{\wr }}}\,{{C}_{2}}  \ltimes \prod_{i=1}^{2^{k-1}-1} C_2  $, where we take $C_2$ as group of action on two elements and this action is faithful, it adjusts with construction of normalizer for $Syl_p(S_n)$ from \cite{Weisner}, where it was said that $Syl_2(A_{2^l})$ is self-normalized  in $S_{2^l}$.

Number of such minimal generating systems for ${{G}_{k}}$ is $C_{{{2}^{k-2}}}^{1}C^1_{{2}^{k-2}}=2^{k-2} \cdot 2^{k-2}= 2^{2k-4}$.

The structure of $Syl_2A_{12}$ is the same as of the subgroup $H_{12} < Syl_2(S_8) \times Syl_2(S_4)$, for which $[Syl_2(S_8) \times Syl_2(S_4):H_{12}]=2$. $|Syl_2(A_{12})|= 2^{[12/2] + [12/4]+ [12/8]-1} = 2^9$. Also $|Syl_2(S_8)|=2^7$, $|Syl_2(S_4)|=2^3$, so $|Syl_2(S_8) \times Syl_2(S_4)|=2^{10}$ and $|H_{12}|=2^9$, because its index in $Syl_2(S_8) \times Syl_2(S_4)$ is 2. The structure of $Syl_2(A_6)$ is the same as of $H_6 < Syl_2(S_4) \times (C_2)$. Here $H_6 = \{(g,h_g)|g \in Syl_2(S_4), h_g \in C_2\}$, where

\begin{equation}\label{H}
 \begin{cases}
h_g = e, \ \ if \ g|_{L_2} \in Syl_2(A_6), \\
h_g = (5,6), \ if \, g|_{L_2} \in {Syl_2(S_6) \setminus Syl_2(A_6)},
 \end{cases}
\end{equation}

The structure of $Syl_2(A_{6})$ is the same as subgroup $H_6:$ $H_6 < Syl_2(S_4) \times (C_2)$ where $H_6= \{ (g, h) | g\in Syl_2(S_4), h \in  AutX \}$. So last bijection determined by (\ref{H}) giving us $Syl_2 A_{6} \simeq Syl_2 S_{4} $. As a corollary we have $Syl_2 A_{{2^k}+2} \simeq Syl_2 S_{2^k} $.
The structure of $Syl_2(A_{7})$ is the same as of the subgroup $H_7:$ $H_7 < Syl_2(S_4) \times S_2$ where $H_6= \{ (g, h) | g\in Syl_2(S_4), h \in  S_2 \}$ and $h$ depends of $g$:

\begin{equation}\label{HH}
 \begin{cases}
h_g = e, \ \ if \  g|_{L_2}\in  Syl_2  A_7, \\
h_g = (i,j), i,j \in\{ 5,6,7 \},  \ if \, g|_{L_2}\in  {Syl_2 S_7\setminus Syl_2A_7},
 \end{cases}
\end{equation}
The generators of the group $H_7$ have the form $(g,h), \, \, g\in Syl_2(S_4), \, h\in C_2$, namely: $ \{ {\beta_{0}; \beta_{1}, \tau} \} \cup \{ (5,6) \}$. An element $h_g$ can't be a product of two transpositions of the set: ${(i,j), (j,k), (i,k)}$, where $i,j,k$  $\in\{ 5,6,7 \} $, because $(i,j)(j,k)=(i,k,j)$ but $ord(i,k,j) =3$, so such element doesn't belong to 2-subgroup. In general elements of $Syl_2 A_{4k+3}$ have the structure (\ref{HH}), where $h_g = (i,j), \,\, i,j \in\{ 4k+1, 4k+2, 4k+3 \}$ and $g\in Syl_2 S_{4k}$.

Also $|Syl_2(S_8)|=2^7$, $|Syl_2(S_4)|=2^3$, so $|Syl_2(S_8) \times Syl_2(S_4)|=2^{10}$ and $|H_{12}|=2^9$, because its index in $Syl_2(S_8) \times Syl_2(S_4)$ is 2. The structure of $Syl_2(A_6)$ is the same as of $H_6 < Syl_2(S_4) \times (C_2)$. Here $H_6 = \{(g,h_g)|g \in Syl_2(S_4), h_g \in C_2\}$.

The orders of this groups are equal, really $|Syl_2(A_7)|= 2^{[7/2] + [7/4]-1}  = 2^3= |H_7|$. In case \, $g|_{L_2}\in  {S_7\setminus A_7}$ we have $C_3^2$ ways to construct one transposition that is direct factor in $H$ which complete $Syl_2 S_4$ to $H_7$ by one transposition  $: \{(5,6); (6,7); (5,7) \}$.

The structure of $Syl_2(A_{2^k+2^l})$ $(k>l)$ is the same as of the subgroup $H_{2^k+2^l} < Syl_2(S_{2^k}) \times Syl_2(S_{2^l})$, for which $[Syl_2(S_{2^k}) \times Syl_2(S_{2^l}):H]=2$. $|Syl_2(A_{2^k+2^l})|= 2^{[(2^k+2^l)!/2] + [(2^k+2^l)!/4]+ .... -1} $.
Here $H = \{(g,h_g)|g \in Syl_2(S_2^k), h_g \in Syl_2(S_2^l\}$, where

\begin{equation}\label{HHH}
 \begin{cases}
h  \in    A_{2^l}, \ \ if \  g|_{X^{k-1}}\in    A_{2^k}, \\
h:  h|_{X^2}\in  {Syl_2 (S_{2^l}) \setminus Syl_2 (A_{2^l})},  \ if \, g|_{X^k}\in  {Syl_2 (S_{2^k}) \setminus Syl_2 (A_{2^k})},
 \end{cases}
\end{equation}
The generators of the group $H_7$ have the form $(g,h), \, \, g\in Syl_2(S_4), \, h\in C_2$, namely: ${\beta_{0}; \beta_{1}, \tau} \cup {(5,6)}$.

I.e. for element  ${{\beta }_{\sigma }}(2i-1)=2\sigma (i)-1,\,\,{{\beta }_{\sigma }}(2i)=2\sigma (i)$, ${{\sigma }_{i}}\in \left\{ {{1,2,...,2}^{k-1}} \right\}$.

\begin{proper} Relation between structures of Sylow subgroups is given by $ Sy{{l}_{2}}({{A}_{4k+2}}) \simeq Sy{{l}_{2}}({S_{4k}})$, where $k\in \mathbb{N}$.
\end{proper}
\begin{proof}
Here $H_6 = \{(g,h_g)|g \in Syl_2(S_4), h_g \in C_2\}$, where

\begin{equation}\label{HHHH}
 \begin{cases}
h_g = e, \ \ if \ g|_{L_k} \in Syl_2(A_{4k+2}), \\
h_g = (4k+1,4k+2), \ if \, g|_{L_k} \in Syl_2(S_{4k+2}) \setminus Syl_2(A_{4k+2}),
 \end{cases}
\end{equation}

The structure of $Syl_2(A_{6})$ is the same as subgroup $H_6:$ $H_6 < Syl_2(S_4) \times (C_2)$ where $H_6= \{ (g, h) | g\in Syl_2(S_4), h \in  AutX \}$. So last bijection determined by (\ref{HHHH}) give us $Syl_2 A_{6} \simeq Syl_2 S_{4} $. As a corollary we have $Syl_2 A_{{2^k}+2} \simeq Syl_2 S_{2^k} $.
\end{proof}
\begin{proper} Relation between orders of Sylows subgroup for $n=4k-2$ and $n=4k$ is given by $\left| Sy{{l}_{2}}({{A}_{4k-2}}) \right|={{2}^{i}}\left| Sy{{l}_{2}}({{A}_{4k}}) \right|$, where value $i$ depend only of power of 2 in decomposition of prime number of $k$.
\end{proper}
\begin{proof}
Really $\left| {{A}_{4k-2}} \right|=\frac{(4k-2)!}{2}$, therefore $\left| {{A}_{4k}} \right|=\frac{(4k-2)!}{2}(4k-1)4k$, it means that $i$ determines only by $k$ and isn't bounded.
 \end{proof}
\begin{lemma} \label{islmorph} If $n=4k+2$, then the group $Syl_2(A_n)$ is isomorphic
to $Syl_2(S_{4k})$.
\end{lemma}

\begin {proof} Bijection correspondence between set of elements of $Syl_2(A_n)$ and $Syl_2(S_{4k})$ we have from (\ref{HHHH}). Let's consider a mapping $\phi: Syl_2 (S_{4k}) \rightarrow Syl_2 (A_{4k+2})$ if $\sigma \in Syl_2(S_{4k})$ then $\phi(\sigma)=\sigma \circ (4k+1, 4k+2)^{\chi(\sigma)}=(\sigma,  (4k+1, 4k+2)^{\chi(\sigma)})$, where $\chi(\sigma)$ is number of transposition in $\sigma$ by module 2.

So $\phi(\sigma) \in Syl_2(A_{4k+2})$.
 If $\phi(\sigma) \in A_{n}$ then ${\chi(\sigma)}=0$, so $\phi(\sigma) \in Syl_2(A_{n-1})$. Check that $\phi$ is homomorphism.
Assume that ${{\sigma }_{1}}\in Sy{{l}_{2}}({{S}_{4k}}\backslash {{A}_{4k}}),\,\,{{\sigma }_{2}}\in Sy{{l}_{2}}({{A}_{4k}})$, then $\phi ({{\sigma }_{1}})\phi ({{\sigma }_{2}})=({{\sigma }_{1}},{{h}^{\chi({{\sigma }_{1}})}})({{\sigma }_{2}},e)=({{\sigma }_{1}}{{\sigma }_{2}},h)={{\sigma }_{1}}{{\sigma }_{2}}\circ (4k+1,4k+2)$, where $({{\sigma }_{i}},h)={{\sigma }_{i}}\circ {{h}^{\chi({{\sigma }_{i}})}}\in Sy{{l}_{2}}({{A}_{4k+2}})$. If ${\sigma _{1}},\,\,{\sigma_{2}}\in {{S}_{{{2}^{k}}}}\backslash {{A}_{{{2}^{k}}}}$, then $\phi ({{\sigma }_{1}})\phi ({{\sigma }_{2}})=({{\sigma }_{1}},{{h}^{\chi({{\sigma }_{1}})}})({{\sigma }_{2}},{{h}^{\chi({{\sigma }_{2}})}})=({{\sigma }_{1}}{{\sigma }_{2}},\,e)=(a,\,e)$, where ${{\sigma }_{1}}{{\sigma }_{2}}=a\in {{A}_{4k+2}}$.
So it is isomorphism.
\end {proof}

\begin{remark}
If $n=4k$, then index $Syl_2(A_{n+3})$ in $A_{n+3}$ is equal to $[S_{4k+1}: Syl_2 (A_{4k+1})](2k+1)(4k+3)$, index
$Syl_2(A_{n+1})$ in $A_{n+1}$ as a subgroup of index $2^{m-1}$, where $m$ is the
maximal natural number, for which $4k!$ is divisible by $2^m$.
\end {remark}
 \begin {proof} For $Syl_2(A_{n+3})$ its order equal to maximal power of 2 which divide $(4k+3)!$ but $(4k+3)!=(4k+1)!(4k+2)(4k+3)=(4k+1)!2(2k+1)(4k+3)$ so $\mid Syl_2 A_{n+3}\mid= 2^m \cdot 2 = 2^{m+1}$.
      Indexes are next: $  [S_{4k+1}: Syl_2 (A_{4k+1}) ] = \frac{(4k+1)!}{2^m} $ -- by condition of Remark, so $[S_{4k+3}: Syl_2 (A_{4k+3})] = [S_{4k+1}: Syl_2 (A_{4k+1}) ](2k+1)(4k+3) = \frac{(4k+1)!}{2^m}(2k+1)(4k+3) $.
\end {proof}

\begin{remark}
If $n=2k$ then
$[Syl_2(A_n) : Syl_2(S_{n-1})] =2^{m-1}$, where $m$ is the maximal power of 2 in factorization of $n$.
\end {remark}

 \begin {proof}
$|Syl_2(S_{n-1})|$ is equal to $t$ maximal power of 2 in $(n-1)!$. $|Syl_2(A_{n})|$ is equal to maximal power of 2 in $(n!/2)$. Since $n=2k$ then $(n/2)!=(n-1)!\frac{n}{2}$ and $2^f$ is equal to product maximal power of 2 in $(n-1)!$ on maximal power of 2 in $\frac{n}{2}$. Therefore $\frac{|Syl_2(A_{n})|}{|Syl_2(S_{n-1})|}=\frac{2^t} {2^{m-1}}{{2^t}}=2^{m-1}. $
Note that for odd $m=n-1$ the group $Syl_2(S_{m}) \simeq Syl_2(S_{m-1})$ i.e. $Syl_2(S_{n-1})\simeq Syl_2(S_{n-2})$. The group $Syl_2(S_{n-2})$ contains the automorphism of correspondent binary subtree with last level $X^{n-2}$ and this automorphism realizes the permutation $\sigma$ on $X^{n-2}$. For every $\sigma\in Syl_2(S_{n-2})$ let us set in correspondence a permutation $\sigma (n-1,n)^{\chi (\sigma)} \in Syl_2(A_{n})$, where $\chi (\sigma)$ -- number of transposition in $\sigma$ by $mod\, 2$, so it is bijection $\phi(\sigma)\longmapsto \sigma (n-1, n){\chi (\sigma)}$ which has property of homomorphism, see Lemma \ref{islmorph}. Thus we prove that $Syl_2(S_{n-1}) \hookrightarrow Syl_2(A_{n})$ and its index is $2^{d-1}$.
\end {proof}

\begin{remark}
For $n=2k+1$ we have the isomorphism
$Syl_2(A_n) \cong Syl_2(A_{n-1})$ and
$Syl_2(S_n) \cong Syl_2(S_{n-1})$.
\end{remark}
 \begin {proof}
 Orders of these subgroups are equal since Legender's formula count power of 2 in $2k+1$ and $2k$ but these powers are equal. So these maximal 2-subgroups are isomorphic. From Statement 1 can be obtained that in vertex with number $2k+1$ will be fixed to hold even number of transpositions on $X^{k_1}$ from decomposition of $n$ which is in Statement 1. For instance $Syl_2(A_{7})\simeq Syl_2(A_{6})$ and by the way $Syl_2(A_{6})\simeq C_2 \wr C_2 \simeq D_4$, $Syl_2(A_{11})\simeq Syl_2(A_{10}) \simeq C_2 \wr C_2 \wr C_2 $.
\end {proof}

\begin{remark}
The ratio of $|Syl_2(A_{4k+3})|$ and $|Syl_2(A_{4k+1})|$ is equal to 2 and ratio of indexes $[A_{4k+3} : Syl_2(A_{4k+3})]$ and $[A_{4k+1} : Syl_2(A_{4k+1})]$ is equal $(2k+1)(4k+3)$.
\end{remark}

\begin{proof} The
ratio of $|Syl_2(A_{4k+3})| : |Syl_2(A_{4k+1})|= 2$ because formula of Legendere give us new one power of 2 in $4k+3$ in compering with $4k+1$.  Second part of statement follows from theorem about $p$-subgroup of $H$, $[G:H] \neq kp $ then one of $p$-subgroups of $H$ is Sylow $p$-group of $G$. In this case $p=2$ but $|Syl_2(A_{4k+3})| : |Syl_2(A_{4k+1})|=2$ so we have to divide ratio of indexes on $2$.
\end{proof}
Let us present new operation $\boxtimes $ (similar to that is in \cite{Dm}) as a subdirect product of $Syl{{S}_{{{2}^{i}}}}$, $n=\sum_{i=0} ^j a_{k_i} {2}^{i}$, i.e. $Syl{{S}_{{{2}^{{{k}_{1}}}}}}\boxtimes Syl{{S}_{{{2}^{{{k}_{2}}}}}}\boxtimes ...\boxtimes Syl{{S}_{{{2}^{{{k}_{l}}}}}}=Par(Syl{{S}_{{{2}^{{{k}_{1}}}}}}\times Syl{{S}_{{{2}^{{{k}_{2}}}}}}\times ...\times Syl{{S}_{{{2}^{{{k}_{l}}}}}})$, where $Par(G)$ -- set of all even permutations of $G$. Note, that $\boxtimes $ is not associated operation, for instance $ord({{G}_{1}}\boxtimes {{G}_{2}}\boxtimes {{G}_{3}})\,\,\,=\left| {{G}_{1}}\times {{G}_{2}}\times {{G}_{3}} \right|:2$ but $ord({{G}_{1}}\boxtimes {{G}_{2}})\boxtimes {{G}_{3}}\,\,\,=\left| {{G}_{1}}\times {{G}_{2}}\times {{G}_{3}} \right|:4$. For cases $n=4k+1$, $n=4k+3$ it follows from formula of Legendre.

\begin {statment}
 Sylow subgroups of ${{S}_{n}}$ is isomorphic to $Syl{{S}_{{{2}^{{{k}_{1}}}}}}\times Syl{{S}_{{{2}^{{{k}_{2}}}}}}\times ...\times Syl{{S}_{{{2}^{{{k}_{l}}}}}}$, where $n={{2}^{{{k}_{1}}}}+{{2}^{{{k}_{2}}}}+...+{{2}^{{{k}_{l}}}}$, ${{k}_{i}}\ge 0$, $k_{i} < k_{i-1}$. Sylow subgroup $Sy{{l}_{2}}({{A}_{n}})$ has index 2 in $Syl_{2}({{S}_{n}})$ and it's structure: $Syl_2{S_{2^{{{k}_{1}}}}}\boxtimes Syl_2{{S}_{{{2}^{{{k}_{2}}}}}}\boxtimes ...\boxtimes Syl_2{S}_{{2}^{{k}_{l}}}$.
 \end {statment}

\begin{proof} Group $Sy{{l}_{2}}({{S}_{{{2}^{k}}}})$ is isomorphic to $\underbrace{{{C}_{2}}\wr {{C}_{2}}\wr ...\wr {{C}_{2}}}_{k}$ \cite{Dm, Sk} and this group is isomorphic to $AutX^{[k]}$, that acts at ${{2}^{k}}$ vertices on ${{X}^{k}}$. In case $Sy{{l}_{2}}({{S}_{n}}),\,\,\, \nexists k,k\in \mathbb{N}:\,\,  n = {{2}^{k}}$ but  $n={{2}^{{{k}_{1}}}}+\,{{2}^{{{k}_{2}}}}+...+{{2}^{{{k}_{l}}}},\,\,{{k}_{i}}\in \mathbb{N}\cup \left\{ 0 \right\}$  is a direct product: $Aut{{X}^{{[{k}_{1]}}}}\times ...\times Aut{{X}^{{[{k}_{l}]}}}$ \cite{Sh}. Power of 2 in this $|A_n|$: $\frac{2^{k_1}}{2}+\frac{2^{k_1}}{4}+...+1+\frac{2^{k_1}}{2}+\frac{2^{k_2}}{4}+...+1+\frac{2^{k_l}}{2}+...+1 - (\underbrace{1 + ... +1}_l)$. But only half of these automorphism determines even permutations on $X^k$ so we have been subtracted 1 $l$ times. For counting order of $Aut{{T}_{{{k}_{1}}}}\times ...\times Aut{{T}_{{{k}_{l}}}}$ we have to take in consideration that $AutX^{[k_i]}$ has active states on levels $X^0$, $X^1$, ... ,$X^{k_i-1}$ and
hasn't active states on $X^{k_i}$ so $|AutX^{[k_i]}|=1\cdot 2 \cdot 2^2 \cdot...\cdot 2^{k_i-1}=2^{2^{k_i}}-1$.

 Such presentations is unique because it determines by binary number presentation  $n={{2}^{{{k}_{1}}}}+\,\,...\,\,+{{2}^{{{k}_{l}}}},\,\,{{k}_{i}}\in \mathbb{N}\cup \left\{ 0 \right\}$, exactly for this presentation corresponds decomposition in direct product $Sy{{l}_{2}}{{S}_{{{2}^{{{k}_{1}}}}}}\times Sy{{l}_{2}}{{S}_{{{2}^{{{k}_{2}}}}}}\times ...\times Sy{{l}_{2}}{{S}_{{{2}^{{{k}_{l}}}}}}$ and this decomposition determines 2-subgroup of maximal order. It is so, since for alternating decomposition ${{2}^{{{k}_{1}}}}={{2}^{l}}+{{2}^{l}}$ order: $\left| Sy{{l}_{2}}{{S}_{{{2}^{{{k}_{1}}}}}} \right|>\left| Syl_{2}{{S}_{{{2}^{l}}}}\times Sy{{l}_{2}}{{S}_{{{2}^{l}}}} \right|$ or more precisely $\left| Sy{{l}_{2}}{{S}_{{{2}^{{{k}_{1}}}}}} \right|=2\left| Sy{{l}_{2}}{{S}_{{{2}^{l}}}}\times Sy{{l}_{2}}{{S}_{{{2}^{l}}}} \right|$ it following from structure of group binary trees.

 For instance $Syl_{2}{{S}_{22}}\simeq Sy{{l}_{2}}{{S}_{16}}\times Sy{{l}_{2}}{{S}_{4}}\times Sy{{l}_{2}}{{S}_{2}}$ and correspondent orders ${{2}^{15}},\,\,{{2}^{3}},\,\,2$ so it's order is ${{2}^{19}}={{2}^{15}}\cdot {{2}^{3}}\cdot 2$, on the other hand order of  $Sy{{l}_{2}}{{S}_{22}}$  by formula of Legendre is ${{2}^{19}}={{2}^{11+5+2+1}}$, analogously $Sy{{l}_{2}}{{S}_{24}}\simeq Sy{{l}_{2}}{{S}_{16}}\times Sy{{l}_{2}}{{S}_{8}}$ and ${{2}^{22}}={{2}^{15}}\cdot {{2}^{7}}$, on the other hand ${{2}^{22}}={{2}^{12+6+3+1}}$. Let us prove that such decompositions of $Sy{{l}_{2}}{{S}_{n}}$ and $Aut{{T}_{n}}$ are unique in accord with $n={{2}^{{{k}_{1}}}}+{{2}^{{{k}_{2}}}}+...+{{2}^{{{k}_{m}}}}$, where ${{k}_{1}}>{{k}_{2}}>...>{{k}_{m}}\ge 0$.

Really in accord with Legender's formula $n!$ is divisible by 2 in power $n-m={{2}^{{{k}_{1}}}}+{{2}^{{{k}_{2}}}}+...+{{2}^{{{k}_{m}}}}-m$ so $\left| Sy{{l}_{2}}{{S}_{22}} \right|={{2}^{n-m}}$, group $Aut{{T}_{{{k}_{1}}}}\times ...\times Aut{{T}_{{{k}_{m}}}}$, where ${{k}_{1}}>{{k}_{2}}>...>{{k}_{m}}\ge 0$, also has order ${{2}^{{{2}^{{{k}_{1}}-1}}+...+{{2}^{{{k}_{m}}-1}}}}={{2}^{n-m}}$ it follows from formula of geometric progression. Such decomposition $Sy{{l}_{2}}{{S}_{{{2}^{{{k}_{1}}}}}}\times Sy{{l}_{2}}{{S}_{{{2}^{{{k}_{2}}}}}}\times ...\times Sy{{l}_{2}}{{S}_{{{2}^{{{k}_{l}}}}}}$ determines 2-Sylows subgroup of maximal order. It is so, since for alternating decomposition ${{2}^{{{k}_{1}}}}={{2}^{l}}+{{2}^{l}}$ order: $\left| Sy{{l}_{2}}{{S}_{{{2}^{{{k}_{1}}}}}} \right|>\left| Sy{{l}_{2}}{{S}_{{{2}^{l}}}}\times Sy{{l}_{2}}{{S}_{{{2}^{l}}}} \right|$ or more precisely $\left| Sy{{l}_{2}}{{S}_{{{2}^{{{k}_{1}}}}}} \right|=2\left| Sy{{l}_{2}}{{S}_{{{2}^{l}}}}\times Sy{{l}_{2}}{{S}_{{{2}^{l}}}} \right|$ it
following from structure of group binary trees ($Aut{{T}_{l}}\times Aut{{T}_{l}}$) which correspondent for $Sy{{l}_{2}}{{S}_{{{2}^{l}}}}\times Sy{{l}_{2}}{{S}_{{{2}^{l}}}}$ and formula of geometric progression. It also follows from nesting (embeding) of Sylows subgroups. If $n={{2}^{{{k}_{1}}}}+{{2}^{{{k}_{2}}}}={{2}^{l}}+{{2}^{l}}+{{2}^{{{k}_{3}}}},$ where ${{2}^{{{k}_{1}}}}={{2}^{l}}+{{2}^{l}}$, then $Aut{{T}_{n}}\simeq Aut{{T}_{{{k}_{1}}}}\times Aut{{T}_{{{k}_{2}}}}$, if not then $Sy{{l}_{2}}{{A}_{n}}$ don't contains $Sy{{l}_{2}}{{A}_{{{2}^{{{k}_{1}}}}}}$.
\end{proof}

We call index of automorphism on $X^l$ number on active states (v. p.) on $X^l$.
\begin {defin} A generator of type \texttt{T} we will call automorphism ${{\tau }_{{{i}_{0}},...,{{i}_{{{2}^{k-1}}}};{{j}_{{{2}^{k-1}}}},...,{{j}_{{{2}^{k}}}}}}$, that has even index at ${{X}^{k-1}}$ and ${{\tau }_{{{i}_{0}},...,{{i}_{{{2}^{k-1}}}};{{j}_{{{2}^{k-1}}}},...,{{j}_{{{2}^{k}}}}}}\in S{{t}_{Aut{{X}^{k}}}}(k-1)$,
and it consists of odd number of active states in vertexes ${{v}_{k-1,j}}$ with number $ j \leq  2^{k-2}$ and odd number in vertexes ${{v}_{k-1,j}}$, ${2^{k-2}} < j \leq {2^{k-1}}$. Set of such elements denote \texttt{T}.

\end {defin}

\begin {defin}
A combined generator is such automorphism ${{\beta }_{{l};\tau }}$, that restriction ${{\beta }_{{l};\tau }}\left| _{{{X}^{k-1}}} \right.$ coincide with ${{\alpha }_{{l}}}$ from ${{S}_{\alpha }}$ and $Rist_{<\beta_{{i_l};\tau }>}(k-1)=\left\langle \tau' \right\rangle $, where $\tau' \in $\texttt{T}. Set of such elements denote \texttt{CG}.
\end {defin}

\begin {defin} A combined element is such automorphism $ {{\beta }_{1,{{i}_{1}};2,{{i}_{2}};...;k-1,{{i}_{k-1}};\tau }}$, that it's restriction
${{\beta}_{1,{{i}_{1}};2,{{i}_{2}};...;k-1,{{i}_{k-1}};\tau}}\left|_{{{X}_{k-1}}} \right.$ coincide with one of elements that can be generated by ${{S}_{\alpha}}$ and $Rist_{<{{\beta}_{1,{{i}_{1}};2,{{i}_{2}};...;k-1,{{i}_{k-1}};\tau'}}>}(k-1)  =\left\langle \tau' \right\rangle $ where $\tau' \in $\texttt{T}. Set of such elements denote \texttt{C}.
\end {defin}
In other word elements $g \in$\texttt{C} and $g' \in$\texttt{CG} on level  ${{X}^{k-1}}$ have such structure as generator of type \texttt{T}.

By distance between vertexes we shall understand usual distance at graph between its vertexes.
By distance of vertex permutations we shall understand maximal distance between two vertexes with active states from $X^{k-1}$.

\begin{lemma} \label{Lemma about keeping of distance} A vertex permutations on ${X}^{k}$ that has distance ${{d}_{0}}$ can not be generated by vertex permutations with distance ${{d}_{1}}:\,\,{{d}_{1}}<{{d}_{0}}$.
\end{lemma}
\begin{proof}
Really vertex permutation ${{\tau }_{ij}}:\,\,\,\,\,\rho ({{\tau }_{ij}})={{d}_{0}},\,\,\,{{d}_{0}}<{{d}_{1}}$ can be mapped by automorphic mapping only in permutation with  distance ${{d}_{0}}$ because automorphism keep incidence relation and so it possess property of isometry. Also multiplication of portrait of automorphism ${{\tau }_{ij}}:\,\,\,\,\,\rho ({{\tau }_{ij}})={{d}_{1}}$ give us automorphism with distance ${{d}_{1}}$, it follows from properties of group operation and property of automorphism to keep a distance between vertexes on graph.
\end{proof}

\begin{lemma} \label{about transposition} An automorphism $\tau $ ( or $\tau{'} \in$\texttt{T} ) can be generated only with using odd number of automorphism from \texttt{C} or \texttt{T}.
\end{lemma}
\begin{proof}
Let us assume that there is no such ${{\tau }_{ij}}$ then accord to Lemma \ref{Lemma about keeping of distance} it is imposable to generate are pair of transpositions ${{\tau }_{ij}}$ with distance $\rho ({{\tau }_{ij}})=2k$ since such transpositions can be generated only by ${{\tau }_{ij}}$ that described in the conditions of this Lemma: $i\le {{2}^{k-2}},\,\,\,j>{{2}^{k-2}}$. Combined element can be decomposed in product ${{\beta }_{{{i}_{l}};\tau }}= \tau \dot {\beta }_{{i}_{l}} $ so we can express from it $\tau$. If we consider product $P$ of even number elements from \texttt{T} then automorphism $P$ has even number of active states in vertexes ${{v}_{k-1,i}}$ with number $ i \leq  2^{k-2}$ so $P$ does not satisfy the definition of generator of type \texttt{T}.
\end{proof}
\begin{corollary} \label{About generating distance} A generator of type \texttt{T} can not be generated by ${{\tau }_{ij}}\in Aut {{v}_{1,1}}{{X}^{[k-1]}}$ and ${{\tau }_{ij}}\in Aut {{v}_{1,2}}{{X}^{[k-1]}}$. I.e. $\tau $ can not be generated by vertex permutations with distance between vertexes less than $2k-2$ as have only elements of \texttt{T} or elements that can be decomposed in product which contains elements of \texttt{T}. The same corollary is true for a combined element.
\end{corollary}
\begin{proof}
It can be obtained from the last two Lemmas \ref{Lemma about keeping of distance} and \ref{about transposition} because such ${{\tau }_{ij}}\in Aut {{v}_{1,1}}{{X}^{[k-1]}}$ does not satisfy conditions of the Lemma \ref{about transposition}.
\end{proof}

It's known that minimal generating system ${{S}_{\alpha}}$ of $Aut{{T}_{k-1}}$ has $k-1$ elements \cite{Gr}, so if we complete it by $\tau$ (or element of type $\tau$) or change one of its generators on $\beta \in \text{T}$ we get system ${{S}_{\beta }}$: ${{G}_{k}}\simeq \left\langle {{S}_{\beta }} \right\rangle $ and $ |{{S}_{\beta }}|= k$.
So to construct combined element we multiply generator ${{\beta }_{i}}$ of ${{S}_{\beta }}$ (or arbitrary element $\beta $ that can be express from ${{S}_{\beta }}$) by the element of type $\tau $, i.e. we take $\tau' \cdot {{\beta }_{i}}$ instead of ${{\beta }_{i}}$ and denote it  ${{\beta }_{i;\tau }}$. It's equivalent that $Ris{{t}_{{{\beta }_{1,i;2,...,j;\tau }}}}(k)=\left\langle \tau' \right\rangle $, where $\tau'$ -- generator of type $\tau $.

\begin{lemma} \label{About not closed set of element of type T} Sets of elements of types \texttt{T}, \texttt{C}  are not closed by multiplication and raising to even power.
\end{lemma}

\begin{proof}
Let  $\varrho, \rho \in$ \texttt{T} (or \texttt{C}) and $\varrho, \rho = \eta$. The numbers of active states in vertexes $v_{k-1, i}$, $1 \leq i \leq 2^{k-2}$ of $\varrho$ and $\rho$ sums by $mod 2$, numbers of active states in vertexes on $v_{k-1, i}$, $1 \leq i \leq 2^{k-2}$ of $\varrho$ and $\rho$ sums by $mod 2$ too. Thus $\eta$ has even numbers of active states on these corteges.
 So $ RiS{{t}_{\left\langle \eta  \right\rangle }}(k-1)$ doesn't contains elements of type $\tau $ so $\eta \notin $\texttt{T}. Really if we raise to even power element ${{\beta }_{1,{{i}_{1}};2,{{i}_{2}};...;k-1,{{i}_{k-1}};\tau }}\in $\texttt{T} or we (calculate) evaluate product of even number multipliers from \texttt{C} corteges ${{\mu }_{0}}$ and ${{\mu }_{1}}$ permutes with wholly subtrees ${{v}_{1,1}}{{X}^{[k-1]}}$ and ${{v}_{1,2}}{{X}^{[k-1]}}$, then we get element $g$ with even index of ${{X}^{k}}$ in ${{T}_{k,0}}$ and ${{T}_{k,1}}$. Thus $g \notin $\texttt{T}. Consequently elements of \texttt{C} don't form a group, set \texttt{T} as a subset of \texttt{C} is not closed too.
\end{proof}
We have to take into account that all elements from \texttt{T} have the same main property to consists of odd number of active states in vertexes ${{v}_{k-1,j}}$ with index $j\le {{2}^{k-2}}$ and odd number in vertexes with index $j:$ ${{2}^{k-2}}<j\le {{2}^{k-1}}$.

Let ${S_{\alpha }}=\left\langle {{\alpha }_{0}},\,{{\alpha }_{1}},...,{{\alpha }_{k-2}} \right\rangle $, $S_{\alpha }^{'}=\left\langle {{\alpha }_{0}},{{\alpha }_{0;1,({i}_{11},{i}_{12});}},...,{{\alpha }_{{0};1,({{i}_{11},...});...,k-2,({i}_{k-2,1},...) }} \right\rangle $, $\left\langle {{S}_{\alpha }} \right\rangle =\left\langle S_{\alpha }^{'} \right\rangle =Aut{{X}^{k-1}}$, $rk\left( {{S}_{a}} \right)=k-1$ where $rk\left( S \right)$ is rank of system $S$ and rank of group \cite{Bog} generated by $S$. Let $S_{\beta }={{S}_{\alpha }}\cup \tau_{i...j} $, $\tau_{i...j} \in$\texttt{T},
${{S}_{\beta }^{'}}=\left\langle
{{\beta }_{0}},{\beta }_{1,({1});\tau^x },..., {{\beta }_{k-2, (1);\tau^x }}, \tau \right\rangle $, $x \in \{0,1\}$, note if $x=0$ then ${\beta }_{{l,({1})};\tau^x} = {\beta }_l$.
$S_{\beta }^{*}=\left\langle
\mathop {\beta}_{0;1,(i_{11},i_{12}); ... ;{m},(i_{m1},...,i_{mj};\tau) }, \ 0<m<k-1, j \leq 2^{m}
  \right\rangle $. In $S_{\beta }^{'}$ can be used $\tau  $ instead of $\tau_{i...j}\in T $ because it is not in principle for proof so let $S_{\beta }=S_{\alpha }\cup \tau $ hence $\left\langle {{S}_{\beta }} \right\rangle ={{G}_{k}}$.
Without loss of generality
  of the proofs we use in $S_{\beta }$ element $\tau$ instead of arbitrary $\tau_{i,...,j} \in$\texttt{T}, also in $S_{\beta }^{*}$ we could use elements of \texttt{C}: ${{\beta }_{{{i}_{j}};\tau }},\,\,\,\,{{\beta }_{{{i}_{_{m}}}}}$ instead more complex elements $\mathop {\beta}_{l_0,l_1,(i_{11},i_{12});...;l_{m},(i_{m1},...,i_{mj};\tau) } $ that can be expressed from  ${{S}_{\beta }^{'}}$ by conjugations and multiplications its elements, note that conjugation by ${\beta }_{l,({1});\tau^x }$ is the same that conjugation by ${\beta }_{l,({1})}$, because all such elements have property of elements from \texttt{C} to have odd number of active states in vertexes ${{v}_{k-1,j}}$ with index $j\le {{2}^{k-2}}$ and odd number in vertexes with index $j:$ ${{2}^{k-2}}<j\le {{2}^{k-1}}$ what will be used in proof, hence $S_{\beta }^{*}$ can be renewed by $S_{\beta }^{'}$. Element ${\beta }_{l,({1+2^{l_1}});\tau^x }$ can be expressed as ${\beta }_{l,({1+2^{l_1}});\tau^x }=\beta_{(k-1)-l_1,({1});\tau^x} {\beta }_{{l,(1)};\tau^x}    (\beta_{(k-1)-l_1,({1});\tau^x})^{-1}$.

 To express element of type \texttt{T} from ${{S}_{\beta }^{'}}$ we have to use a word ${{\beta }_{i,\tau }}\beta _{i}^{-1}=\tau $ but if ${{\beta }_{i,\tau }}\in {{S}_{\beta }^{'}}$ then ${{\beta }_{{{i}}}}\notin S_{\beta }^{'}$ so we have to find relation in restriction of group ${{G}_{k}}$ on ${{X}^{k-1}}$ to express word ${{\beta }_{i,\tau }}\beta _{i}^{-1}\left| _{{X}^{k-1}} \right.=e$ then ${{\beta }_{i,\tau }}\beta _{i}^{-1}=\tau $. Next lemma investigate existing  of such word.  We have to take in consideration that $ {{G}_{k}}\left| _{{{X}^{k-1}}} \right.={{B}_{k-1}}\simeq Aut{{X}^{k-1}} $. Let us assume that $\beta_{i_0}\in S_{\beta }^{'}$ for arbitrary $0 <i_0 <k$.

\begin{lemma} \label{Lem About product elem of C}
An element of type \texttt{T} cannot be expressed by $S_{\beta }^{'}={{S}_{\beta }}\cup \{{{\beta }_{{{i}_{_{0}}};\tau }}\}\backslash \{{{\beta }_{{{i}_{_{0}}}}},\tau \}$ with using product, where are odd number of the elements of combine type.
\end{lemma}

\begin{proof}
Let's, show that an element of type $\tau $ cannot be expressed from $S_{\beta }^{'}$ without using in product of generators even number of combine generators, because word ${{\alpha }_{{{i}_{1}}}}{{\alpha }_{{{i}_{2}}}}...{{\alpha }_{{{i}_{n}}}}=e$ of $Aut{{T}_{k-1}}$ is word that may be reduced to relation in $\underbrace{{{S}_{2}}\wr ...\wr {{S}_{2}}}_{k-1}$, which studied in \cite{DrSku, Sk}.
We can express element of type $\tau $ only if we use elements from \texttt{C}. Note that element of type $\tau $ belongs to $Ri\text{S}{{\text{t}}_{Aut{{T}_{k}}}}(k-1)$. Let us show that automorphism $\alpha :\,\,\,\alpha \left| _{{{X}^{k-1}}} \right.\in ~Ri\text{S}{{\text{t}}_{G}}(k)$, i.e. that is trivial at restriction on ${X}^{k-1}$ and belongs to
\texttt{T} can not be expressed from $S_{\beta }^{'}$ without using even number of elements from \texttt{C} because word ${{\alpha }_{{{i}_{1}}}}{{\alpha }_{{{i}_{2}}}}...{{\alpha }_{{{i}_{n}}}}=e$ of $Aut{{T}_{k-1}}$ is word that may be reduced to relation in $\underbrace{{{S}_{2}}\wr ...\wr {{S}_{2}}}_{k-1}$. It known that in wreath product $\wr _{1}^{k}{{\mathcal{C}}_{n}}$ holds a relations $\beta _{i}^{m}=e$ and $[\beta _{i}^{-1}{{\beta }_{j}}{{\beta }_{i}},\sigma _{j}^{{}}]=e,\,\,i<j$, $\left\langle {{\beta }_{j}},{{\sigma }_{j}} \right\rangle \simeq {{G}_{j}}$, \cite{Sk, DrSku} or more specially $\left[ \beta _{m}^{i}{{\beta }_{{{i}_{n}}.\tau }}\beta_{m}^{-i},\,\beta _{m}^{j}{{\beta }_{{{i}_{k}}.\tau }}\beta _{m}^{-j} \right]=e,\,\,\,i\ne j$, where $n,k,m$ are number of group in $\wr _{1}^{k}{{\mathcal{C}}_{l}}$ ($m<n$, $m<k$) \cite{Sk}, ${\beta }_{{{i}_{n}}}$, ${{\beta }_{m}},\,\,{{\beta }_{{{i}_{k}}}}$ is generator from ${{S}_{^{\beta }}}$.
As we see such relations $r_i$ has structure of commutator so logarithm \cite{K} of every $r_i$ by every generator $\beta_i \in S_\beta$ is zero because every exponent entering with different signs so every element from \texttt{C} has logarithm 0 hence is not equal to element of T because Lemma \ref{About not closed set of element of type T}.
 Really product of 2 arbitrary elements of \texttt{T} (analogously  \texttt{C}) isn't element of \texttt{T}.
So other way to express automorphism $\alpha :\,\,\,\alpha \left| _{{{X}^{k-1}}} \right.\in ~Ri\text{S}{{\text{t}}_{{{B}_{k-1}}}}(k-1)$ doesn't exist according to relations in ${{G}_{k}}$. Cyclic relation in this group has form $\beta _{{{i}_{0}};\tau }^{{{2}^{x}}}=e$, analogously for $\beta _{1,{{i}_{1}};2,{{i}_{2}};...;k-1,{{i}_{k-1}};\tau }^{2y}=e$ (analogously  for $\beta _{1,{{i}_{1}};2,{{i}_{2}};...;k-1,{{i}_{k-1}};\tau }^{2y}=e$) because it is 2-group, so we need to raise generator to even power.
So other way to express automorphism $\alpha :\,\,\,\alpha \left| _{{{X}^{k-1}}} \right.\in ~Ri\text{S}{{\text{t}}_{{{B}_{k-1}}}}(k-1)$ does not exist according to relations in ${{G}_{k}}$.
\end{proof}

\begin{lemma}\label{corresp} The systems ${{S}_{\alpha }}$ and $S_{\alpha}^{'} $, ${{S}_{\beta }}$ and $S_{\beta }^{'}$ are in bijective correspondence.
\end{lemma}
\begin{proof}
As is well known, for binary tree the rank of $Aut{{X}^{k-2}}$ is  $k-1$ \cite{Gr}, so for the elements of other least system of generators $S_{\alpha }^{'}$ the elements of ${{S}_{\alpha }}$ are in bijective correspondence. Any element of other system $S_{\alpha }^{'}$ has type ${{\alpha }_{0,({{i}_{0}});1,({{i}_{11}},{{i}_{12}});...;m,({{i}_{m1}},...,{{i}_{mJ}})}},\,\,m\le k-2$ and elements of ${{S}_{\alpha }}$: ${{\alpha }_{0}},{{\alpha }_{1}},{{\alpha }_{2}}\,,...,\,\,{{\alpha }_{k-2}}$.
Set $S_{m}^{l}$ of generators from $S_{\alpha }^{'}$ each of which has non trivial state on ${{X}^{l}}$ is such that generators from $S_{m}^{l}$ has at least one active state from one of $m$ different levels. So elements of $S_{m}^{l}$ has $m$ different preimages in ${{S}_{\alpha }}$. This is Hall's condition for existence of the pairmaching, this condition is hold because these systems have the same rank \cite{Bog}. Therefore we have bijection between ${{S}_{\alpha }}$  and $S_{\alpha }^{'}$. We construct ${{S}_{\beta }}$: ${{S}_{\beta }}={{S}_{\alpha }}\cup \tau $.
Let us show bijective correspondence between ${{S}_{\beta }}$ and $S_{\beta }^{'}$. To establish this correspondence we use that elements of $S_{\beta }^{'}$ are determed through ${{\beta }_{i,(1)}}={\alpha }_{_{i,(1)}}={\alpha }_{i} \in S_{\alpha }$ for some elements of $S_{\alpha }$
and  elements of \texttt{C} from $S_{\beta }^{'}$ are determed: ${{\beta }_{j,(1);\tau }}=\tau {{\alpha }_{j}}$.
  And $S_{\beta }^{*}$ is determed by $S_{\alpha }^{'}$
 through ${{\beta }_{0,({i}_{0}); ... ;m,({{i}_{m1}},...,{i}_{mJ})}}={{\alpha }_{0,({{i}_{0}}); ... ;m,({{i}_{m1}},...,{{i}_{mJ}})}} \in S_{\alpha }^{'}$
  for some elements of $S_{\beta }^{'}$ and elements of \texttt{C} from $S_{\beta }^{*}$ determ by ${{\beta _{0,({{i}_{0}});1,({{i}_{11}},{{i}_{12}});...;m,({{i}_{m1}},...,{{i}_{m2^m}});\tau}
    ^{'} }}=\tau {\alpha }_{0,({{i}_{0}});1,({{i}_{11}},{{i}_{12}});...;m,({i}_{m,1},...,{{i}_{mJ}})}$, $j\le {{2}^{k-1}}$, ${{\beta }_{i,(1);\tau }}\in C$. Element ${{\tau }_{i...j}}$ of type \texttt{T} can be expressed from ${{S}_{\beta }}={{S}_{\alpha }}\cup \tau$ that is proved in Lemma \ref{gen}. So we can get every ${{\beta }_{i,(1);{{\tau }_{i...j}}}}={{\tau }_{i...j}}{{\alpha }_{i}}$ and every ${{\beta }_{0,({{i}_{0}});1,({{i}_{11}},{{i}_{12}});...;m,({{i}_{m,1}},...,{{i}_{mJ}});{{\tau }_{i...j}}}}$, ${{\beta }_{0,({{i}_{0}});1,({{i}_{11}},{{i}_{12}});...;m,({{i}_{m,1}},...,{{i}_{mJ}})}}\in S_{\alpha }^{'}$.
\end{proof}



\begin{corollary}\label{generating pair} A necessary and sufficient condition of expressing an element  $\tau $ from $S_{\beta }^{'}$  is existing of pair:    ${{\beta }_{{{i}_{m}};\tau }},\,\,\,\,{{\beta }_{{{i}_{_{m}}}}}$ in $\langle S_{\beta}^{'} \rangle$. So rank of a generating system of $G_k$ which contains $S_{\beta }^{'}$ is at least $k$.
\end {corollary}
\begin{proof} Proof can be obtained from Lemma \ref{Lem About product elem of C} and Lemma \ref{About not closed set of element of type T}  from which we have that element of type \texttt{T} cannot be expressed from $\left[ \beta _{m}^{i}{{\beta }_{{{i}_{n}}.\tau }}\beta _{m}^{-i},\,\beta _{m}^{j}{{\beta }_{{{i}_{k}}.\tau }}\beta _{m}^{-j} \right]=e,\,\,\,i\ne j$ because such word has even number of element from \texttt{C}. Sufficient condition follows from  relation  ${{\beta }_{{{i}_{m}};\tau }}\beta _{{{i}_{_{m}}}}^{-1}=\tau $.
\end{proof}

\begin{lemma} \label{rk} A system of generators of  ${{G}_{k}}$ contains $S_{\alpha}^{'}$ and has at least $k-1$ generators.
\end{lemma}
\begin{proof} Subgroup ${{B}_{k-1}}<{{G}_{k}}$ is isomorphic  to $Aut{{X}^{k-1}}$ that has a  least  system of generators  of $k-1$ elements \cite {Gr}. Moreover, the subgroup ${{B}_{k-1}}$ is a quotient group ${}^{{{G}_{k}}}/{}_{{{W}_{k-1}}}$, because  ${{W}_{k-1}}\triangleleft {{G}_{k}}$. As is well known that $\text{rk}(G)\ge \text{rk}(H)$, because all generators of $G_{k}$ may contain in different quotient classes \cite{Magn}.
\end{proof}

Let us assume that exists a system of generators $S_{\beta }^{''}$ for ${{G}_{k}}$, such that $rk(S_{\beta }^{''}) = k-1$.
Let ${{\beta }_{i,\tau }},\,\,{{\beta }_{i,\tau '}}$ are restrictions of automorphisms  of binary tree  on ${{X}^{k}}$ and restrictions of them on ${{X}^{k-1}}$. Let's denote them as ${{\alpha }_{i}}={{\beta }_{i,\tau }}\left| _{{{X}^{k-1}}} \right.=\,\,{{\beta }_{i,\tau '}}\left| _{{{X}^{k-1}}} \right.$.

\begin {remark} \label{doesnt contain} A system $ {{S}^{''}_{\beta }}$ does not contain elements ${{\beta }_{i}}$ and ${{\beta }_{i,\tau }}$ simultaneously.
\end {remark}

\begin{proof}
  Account Lemma \ref{rk} we have  $rk( S_{\alpha }^{'} )   \geq  k-1$ and $\langle S_{\alpha }^{'} \rangle = B_{k-1}$, $B_{k-1}$ as quotient group
  is embed in $G_k$. But Lemma \ref{Lemma about keeping of distance} leading to it remains to express an element from \texttt{T} to generate $G_k$.  And according to Lemma \ref{rk} $S_{\beta }^{''} \supseteq S_{\alpha}^{'} $ so $S_{\beta }^{''}$ already has $k-1$ generators.

If we converse condition of remark i.e. $\{{{\beta }_{i}}, {{\beta }_{i,\tau }}\}\in S_{\beta }^{''}$, then we can express $\tau  ={\beta }_{i,\tau } \cdot \beta _{i}^{-1}$ because $\langle  S_{\beta }^{''} \rangle = G_k $, where $\beta_{i}  \in {S}^{'}_{\alpha }$ for arbitrary $1 < i < k-1$. But $  {S}^{''}_{\alpha} $ together with $\beta _{i,\tau }$ has $k$ generators that is contradiction with $rk(S_{\beta }^{''})=k-1$.
If such elements exist that could be expressed ${{\beta }_{i,\tau }},\,\,{{\beta }_{i,\tau '}}$,  then them corresponds reduced words ${{\beta }_{i,\tau }}\left| _{{{X}^{k-1}}} \right.={{\beta }_{{{i}_{1}}}}{{\beta }_{{{i}_{2}}}}...{{\beta }_{{{i}_{m}}}}=\,\,{{\beta }_{{{j}_{1}}}}{{\beta }_{{{j}_{2}}}}...{{\beta }_{{{j}_{m}}}} = {{\beta }_{i,\tau '}}\left| _{{{X}^{k-1}}} \right.$, where ${{\beta }_{{{i}_{l}}}},{{\beta }_{{{j}_{k}}}}\in$\texttt{C} in particular case ${{\beta }_{{{j}_{1}}}}{{\beta }_{{{j}_{2}}}}...{{\beta }_{{{j}_{m}}}}\left| _{{{X}^{k-1}}} \right.={{\beta }_{i}}$.
And ${{\beta }_{{{i}_{1}}}}{{\beta }_{{{i}_{2}}}}...{{\beta }_{{{i}_{m}}}}\left| _{{{X}^{k-1}}} \right.={{\alpha }_{{{i}_{1}}}}{{\alpha }_{{{i}_{2}}}}...\,{{\alpha }_{{{i}_{m}}}}=\,\,{{\alpha }_{{{j}_{1}}}}{{\alpha }_{{{j}_{2}}}}...\,{{\alpha }_{{{j}_{l}}}}=\,{{\beta }_{{{j}_{1}}}}{{\beta }_{{{j}_{2}}}}...\,{{\beta }_{{{j}_{l}}}}\left| _{{{X}^{k-1}}} \right.$ so it could be expressed relation ${{\alpha }_{{{i}_{1}}}}{{\alpha }_{{{i}_{2}}}}...\,{{\alpha }_{{{i}_{m}}}}\alpha _{{{j}_{l}}}^{-1}\alpha _{{{j}_{2}}}^{-1}...\,\alpha _{{{j}_{1}}}^{-1}=e$. Because this is a relation from ${{S}_{2}}\wr {{S}_{2}}\wr ...\wr {{S}_{2}}$, then it has type of commutator  or product of reciprocally inverse elements or generator in power ${{2}^{s}}$ and their products. So number of generators in such relation are even. Restriction on ${{X}^{k}}$ of this elements give us product no equal to $e$ so at least one generator  ${{\alpha }_{i}}$ from word ${{\alpha }_{{{i}_{1}}}}{{\alpha }_{{{i}_{2}}}}...\,{{\alpha }_{{{i}_{m}}}}\alpha _{{{j}_{l}}}^{-1}\alpha _{{{j}_{2}}}^{-1}...\,\alpha _{{{j}_{1}}}^{-1}=e$ have different restrictions on ${{X}^{k}}$ in different  entering  of ${{\alpha }_{i}}$ in relation that were obtained, in other words this generator has different images in
 $S_{\beta }^{''}$  which could be obtained via ${{\beta }_{i,\tau }}=\tau {{\alpha }_{i}}$(Lemma 9 about bijection). So such $S_{\beta }^{''}$ has rank on 1 more then ${{S}_{\alpha }}$ what is contradiction.
\end{proof}
As a corollary of Remark \ref{doesnt contain} we see that such $S_{\beta }^{''}$ of rank $k-1$ does not exist.



It is well known that for finite 2-group $G$: $G' < G^2$, so $G'G^2 = G^2$.

\begin{statment} \label{comm}
Frattini subgroup $ \phi(G_k)= {{G_k}^{2}}\cdot [G_k,G_k]= {{G_k}^{2}} $ acts by all even permutations on ${{X}^{l}},\,\,\,0\le l<k-1$
and by all even permutations on $ X^{k}$ except for those  from \texttt{T}.
\end{statment}

\begin{proof}
Index of automorphism $\alpha^2 $, $\alpha  \in <{{S}_{\beta }}>$ on $X^l$ is always even. Really the parity of the number of vertex permutations at $X^l$ in the product $({\alpha }_i {\alpha }_j)^2$, ${\alpha }_i, {\alpha }_j \in  S_{\alpha }$, $i,j<k$ is determined exceptionally by the parity of the numbers of vertex permutations at ${{X}^{l}}$ in $\alpha $ and $\beta $ (independently of the action of v.p. from the higher levels). On $X^{k-1}$ group $G^2$ has all transpositions of form $\tau_{1 i+1},  i\leq 2^{k-1}$ which can be generated in such way $({\alpha }_{k-2} \tau_{12})^2= \tau_{1234}$, $\tau_{12}\tau_{1234} = \tau_{34}$,  $({\alpha }_{k-i} \tau_{12})^2= \tau_{1, 2, 1+2^{k-i}, 2+2^{k-i}}$ then $\tau_{1, 2, 1+2^{k-i}, 2+2^{k-i}} \tau_{12} =\tau_{1+2^{k-i}, 2+2^{k-i}}$. In such way we get set of form ${\tau_{12}, \tau_{23},, \tau_{34}, ... ,\tau_{2^{k-1}-1,2^{k-1}}}$. This set of transpositions is base for $W_{k-1}$.

Let us consider ${{\left( {{\alpha }_{0}}{{\alpha }_{l}} \right)}^{2}}={{\beta }_{;({{1,2}^{l-1}}+1)}}$. Conjugation by element ${{\beta }_{1(1,2)}}$ (or ${{\beta }_{i(1,2)}},\,\,0<i<l$) give us ability to express arbitrary coordinate $x:\,\,1\le x \leq {{2}^{l-1}}$ i.e. from ${{\beta }_{l({{1,2}^{l-1}}+1)}}$ we express ${{\beta }_{l(x{{,2}^{l-1}}+1)}}$, for instance  $x={{2}^{j-1}}+1$, $j<l$: ${{\beta }_{l-j(1,2)}}{{\beta }_{l({{1,2}^{l-1}}+1)}}{{\beta }_{l-j(1,2)}}={{\beta }_{l({{2}^{j-1}+1},{{2}^{l-1}}+1)}}$. If $x={{2}^{l-j}}+2$ it should to conjugate by ${{\beta }_{l-j(1,2)}} {{\beta }_{l-1(1,2)}}$, by such conjugations can be realized every shift on $x$ as 2-adic number. So in such way can be realized every ${{\beta }_{l(x{{,2}^{l-1}}+1)}}$ and analogously every  ${{\beta }_{l({{2}^{l-1}})},y}$ and ${{\beta }_{l(x,y)}}$. Hence we can express from elements of $G^2$ every even number of active states on $X^l$.

In the product  (of portraits) $\alpha ,\,\beta \in {{S}_{\alpha }}$  cortege of v. p. of
${{X}^{l}}$, $0\le i\le k-2$ is determined only by  parity of corteges from ${{X}^{l}}$ of $\alpha $ and $\beta$. So $[\alpha ,\beta ]=\alpha \beta {{\alpha }^{-1}}{{\beta }^{-1}}$ is permutation of even power on ${{X}^{l}}$, $0\le i\le k-2$ thus  commutators of generators from ${{S}_{\alpha }}$ generate permutation of even power. Really, let us generate by commutators of even permutations of ${{X}^{1}}$, let consider commutator $\left[ {{\alpha }_{0}},{{\alpha }_{1}} \right]={{\beta }_{0,1(1,1)}}$, trivial v.p. can be expressed by $\left[ {{\alpha }_{0}},{{\alpha }_{0}} \right]=e$.
The parity of the number of vertex permutations at $X^l$ in the product ${\alpha }_i {\alpha }_j$, ${\alpha }_i,{\alpha }_j \in  S_{\alpha }$) is determined exceptionally by the parity of the numbers of vertex permutations at ${{X}^{l}}$ in $\alpha $ and $\beta $ (independently of the action of v.p. from the higher levels). Thus $[\alpha ,\beta ]=\alpha \beta {\alpha }^{-1}{\beta }^{-1}$ has an even number of v. p. at each level. Therefore, the commutators of the generators from ${{S}_{\alpha }}$ generate only the permutations with even number of v. p. at each ${{X}^{l}}$, ($0\le l\le k-2$). Taking the products of all commutators with index $l$ we can express all automorphisms those induce all even permutations on ${{X}^{l}}$.
\end{proof}
As a corollary we get that $(Aut{{X}^{[l]}})'$ has all even indexes on each level $l\leq k-2$.


Structure of subgroup $G_{k}^{2}{{G}_{k}}'\triangleleft \underset{1}{\overset{k}{\mathop{\wr }}}\,{{S}_{2}}\simeq Aut{{X}^{[k]}}$ can be described in next way. This subgroup contains the commutant -- ${G}_{k}'$ so it has on each ${{X}^{l}},\,\,\,0\le l<k-1$ all even indexes that can exists in vertexes of ${{X}^{l}}$. On ${{X}^{k-1}}$ does not exist v.p. of type \texttt{T}  that is, which has distance $2k-2$, rest of even indexes are present on ${X}^{k-1}$. It's so, because the sets of elements of types \texttt{T} and \texttt{C} are not closed under operation of taking the even power. Thus, the squares of the elements don't belong to \texttt{T} and \texttt{C} (because they have the distance, which is less than $2k-2$). That's why their products also aren't the elements of types \texttt{T} and \texttt{C}, because, according to
 Lemma 5, the elements with the distance of v. p. less than $2k-2$ doesn't generate the elements with distance $2k-2$ after multiplication.
To show that all even permutations on ${{X}^{k}}$  can be expressed besides those automorphisms from \texttt{T} and \texttt{C}  let us consider $\beta _{0,k-1({{1,2}^{k-1}})}^{2}=\beta _{k-1({1,2^{k-2}},{2^{k-2}}+1,\,{{2}^{k-1}})}\notin $\texttt{C} and denote that this automorphism has index 4 on ${{X}^{k-1}}$ and consider $\beta _{1(1),k-1({1,2^{k-2}}+1)}^{2}=\beta _{k-1({1,2^{k-2}})}\notin $\texttt{C}  analogously for $\beta _{j(1),k-1({{1,2}^{k-1}}+1)}^{2}\,,j<k-1$. Note that $\beta _{k-1 ({{1,2}^{k-1}})}^{{}}\notin $ \texttt{C} but ${{\beta }_{0,k-1({{1,2}^{k-2}}+1)}}\in $C. If we change $j$ we can generate set of all even permutations (via multiplication of squares), which presents all automorphism from ${{G}_{k}}$ besides that from \texttt{T} and \texttt{C}, this set of permutations forms normal subgroup in ${{G}_{k}}$.
This implies the following corollary.

\begin{corollary} A quotient group ${}^{{{G}_{k}}}/{}_{G_{k}^{2}{{G}^{'}_{k}}}$ is isomorphic to $\underbrace{{{C}_{2}}\times {{C}_{2}}\times ...\times {{C}_{2}}}_{k}$.
\end{corollary}
\begin{proof}
Proof follows from that $G_{k}^{2}{{G}^{'}_{k}}\simeq G_{k}^{2}\triangleleft {{G}_{k}}$ and $\left| G:G_{k}^{2}{{G}^{'}_{k}} \right|=2^k$.  Restriction  ${{\left. {{G}_{k}} \right|}_{{{X}^{k}}}}$ acts only by permutations of 2 types by parity of permutation on sets ${{X}_{1}}=\{{{v}_{k,1}},...,{{v}_{k{{,2}^{k-1}}}}\}$ and ${{X}_{2}}=\{{{v}_{k,2^{k-1}+1}},...,{{v}_{k{{,2}^{k}}}}\}$, ${{X}_{1}}\cup {{X}_{2}}={{X}^{k}}$. Type 1, where on ${{X}_{1}}$ and ${{X}_{2}}$ $G_k$ acts by odd permutations
(it means that on set vertexes of $X^{k-1}$ over ${{X}_{1}}$ i.e. which joined by edge vertexes of $X^{k-1}$ over ${{X}_{1}}$ (and ${{X}_{2}}$ analogously) automorphism of ${{G}_{k}}$ has odd numbers of active states).
And type 2 (elements of which form a normal subgroup $G_k(2)$ in $G_k$), where on ${{X}_{1}}$ and ${{X}_{2}}$ this permutations act by all even permutations. Subgroup $G_{k}^{2}{{G}^{'}_{k}}$ as $G_k(2)$ acts only by permutations of type 2 on ${{X}_{1}}$, ${{X}_{2}}$, because
     of Statement 2.
      These types of resulting permutations $\rho_1 \rho_2 \left|_{{{X}^{l}}} \right. , \rho_1, \rho_2 \in G_{k}$ does not depend of portrait of automorphism on higher levels ${{X}^{l}},\,\,\,0\le l <k-1$
 because parity of result depends only from parity of summands  and does not depends of permutations of active states of
${{X}^{l}}$. So ${{\left. {}^{{{G}_{k}}}/{}_{G_{k}^{2}{{G}^{'}_{k}}} \right|}_{{{X}^{1}}}}$  has 2 quotient classes
because acts as $S_2$ on ${{X}_{1}}$ but $G_{k}^{2}{G}_{k}^{'}$ acts only by even permutations on ${{X}_{1}}$.
 Index of normal subgroup $  G_{k}^{2}{{G}_{k}^{'}} \left|_{{{X}^{2}}} \right.$ of group ${G}_{k}$ is $2^2$
 because $G_k$ acts on $X^2$ by $2^3$ permutations and $G_{k}^{2}{{G}_{k}^{'}}$ acts by 2 permutations so $|G_{k} \left|_{{{X}^{2}}} \right. : G_{k}^{2}{{G}_{k}^{'}} \left|_{{{X}^{2}}} \right. |=2^2$.
Analogously we have $2^i$ quotient classes for restriction on each level ${{X}^{i}}$, $0<i \leq k-2$. Index of every automorphism of $X^{k-2}$ from $B^{'}_{k-1}\simeq (AutX^{[k-1]})^{2}$ is even so ${}^{St_{B_{k-1}}(k-2)}\left|_{X^{[k-1]}} \right./{}_{St_{B^{2}_{k-1}}(k-2)}\left|_{X^{[k-1]}} \right. \simeq C_2$.
We could count this quotient classes in other way by constructing a homorphism similar to \cite{Gr} in ${C}_{2}$ from cortege of v.p. of level ${X}^{l}$ which forming a subgroup $G_k(X^l)$ as sum of non trivial v.p. by $\bmod 2$ from $X^l$, where $St_{G^2_k}(X^l)\left|_{X^{l}} \right. \lhd St_{G_k}(X^l)\left|_{X^{l}} \right. = G_k(X^l)$, $| St_{G_k}(l)\left|_{X^{[l]}} \right. : St_{G^2_k}(l) \left|_{X^{[l]}} \right. | = 2$. But for subgroup of v.p. on ${X}^{k-1}$ we construct a homorphism in ${C}_{2}$ as a sum of non trivial v.p. by $\bmod 2$ on each set ${X}_{1}$ and ${X}_{2}$.
For a subgroup of v.p. from ${{X}^{k-1}}$ (denote it ${{G}_{k}}({{X}^{k-1}})$, ${{G}_{k}}({{X}^{k-1}})\triangleleft {{G}_{k}}$) construct a homomorphism: $\rho \left( {{G}_{k}}({{X}_{k}}) \right)\to {{C}_{2}} \simeq {}^{{{G}_{k}}({{X}_{k}}})/{}_{ G_k(2)} $ as product of sum by $mod2$ of active states (${{s}_{ij}},\,i\in \{0,1\}$, $0<j\le {{2}^{k-2}}$ if ${{s}_{ij}}\in {{X}_{1}}$)
   on each set ${{X}_{1}}$ and ${{X}_{2}}$:  $\phi ({{X}_{1}})\,\,\bullet \,\,\phi ({{X}_{2}})=\sum\limits_{i=1}^{{{2}^{k-2}}}{{{s}_{k-1,i}}}(\bmod 2) \bullet \sum\limits_{i={{2}^{k-2}}+1}^{{{2}^{k-1}}}{{{s}_{k-1,i}}} (\bmod 2)$.
  Where ${{s}_{k-1,i}}=1$ if  there is active state  in ${{v}_{k-1,i}}$, ${{s}_{k-1,i}}=0$  if there is no active state. From structure of ${{G}_{k}}$ follows that $\phi ({{X}_{1}})\,\,=\,\,\phi ({{X}_{2}})$ so 0 or 1. But $G_{k}^{2}{{G}^{'}_{k}}$ admits only permutations of Type 2 on ${{X}^{k}}$ so $G_{k}^{2}{{G}_{k}}'({{X}_{k}})\triangleleft {{G}_{k}}({{X}_{k}})$ because it holds conjugacy and kernel of map from ${{G}_{k}}({{X}_{k}})$ to ${{C}_{2}}$ is $G_{k}^{2}{{G}_{k}}'({{X}_{k-1}})$.

Consider subgroup of v.p.  for arbitrary $0\le l<k-1$, ${{X}^{l}}$ (denote it ${{G}_{k}}({{X}^{l}})$, ${{G}_{k}}({{X}^{l}})\triangleleft {{G}_{k}}$).  A number of v.p. in ${{G}_{k}}({{X}^{l}})$ can be as even as odd. But number of v.p. in $G_{k}^{2}{G^{'}_{k}}(X^l)$ is even so $G_{k}^{2}{{G}_{k}}'({{X}^{l}})\triangleleft \,{{G}_{k}}({{X}^{l}})$ and there is 2 factor classes in ${}^{{{G}_{k}}({{X}^{l}})}/{}_{G_{k}^{2}{{G}^{'}_{k}}({{X}^{l}})}$ so it exists homomorphism from ${{G}_{k}}({{X}^{l}})$ in ${{\text{C}}_{2}}$.  Hence there is homomorphism from ${{G}_{k}}$ to $\prod\limits_{i=1}^{k}{{{\text{C}}_{2}}}$. So $\mid {{G}_{k}} : G_{k}^{2}{G^{'}_{k}} \mid = 2^k$ and correspondent quotient group is isomorphic to $\prod\limits_{{}}^{k}{{{C}_{2}}}$.
\end{proof}
   \begin{corollary} Group ${{G}_{k}}$ (${{G}_{k}}\simeq A_{2^k}$) has least system of generators from $k$ elements.
\end{corollary}
   \begin{proof}
Since quotient group of ${{G}_{k}}$ by subgroup of Fratiny $G_{k}^{2}{{G}^{'}_{k}}$  has least system of generators from $k$ elements because ${}^{{{G}_{k}}}/{}_{G_{k}^{2}{{G}^{'}_{k}}}$  is isomorphic to linear $p$-space $(p=2)$ of dimension $k$ (or elementary abelian group) then ${{G}_{k}}$ has rank $k$  \cite{Rot}.
   \end{proof}
 \begin{main_theorem}
The set $S_{\mathop{\beta}}=\{\mathop{\beta}_{0}, \mathop{\beta}_{1}, \mathop{\beta}_{2}, \ldots , \mathop{\beta}_{k-2}, \tau \}$, where $\mathop{\beta}_{i} = \alpha_i$, 
 is a minimal generating system for a group $G_k$ that is isomorphic to Sylow 2-subgroup of $A_{2^{k}}$.
\end{main_theorem}
We have isomorphism of $G_k$ and $Syl_2 (A_{2^k})$ from Theorem \ref{isomor}, the minimality of $S_{\mathop{\beta}}$ following from
Lemma \ref{Lem About product elem of C} which said that $S_{\mathop{\beta}}$ has to contain an element of type \texttt{T}. If the system has form ${{S}}_{\beta }^{'}$ then, according to Corollary \ref{generating pair}
 ${{{S}'}_{\beta }}$ has to compleated by elements: ${{\beta }_{{{i}_{m}};\tau }},\,\,\,\,{{\beta }_{{{i}_{_{m}}}}}$ or elements from which these elements can be expressed to get eltment $\tau$ (or element from T), so a new system has at least $k$ elements. The necessity of $\tau$ or element of type \texttt{T} follows from Lemma \label{about transposition} and Corollary \ref{About generating distance}. So if we take in consideration Corollary \ref{generating pair} and corollary from Remark \ref{doesnt contain}, then we see that $rk(Syl_2(A_{2^k}))=k$.

For example a minimal system of generators for $Syl_2(A_{8})$ can be constructed by following way, for convenience let us consider the next system:





\shorthandoff{"}
\def\vertex{\scriptscriptstyle\cdot}
\def\objectstyle{\scriptstyle}
\xy <1cm,0cm>:
(10,3.5)="C0"*{\ }+<-1.2ex,-0.45ex>*{ \beta_0  },
(11,3)="C1"*{\vertex}+<-1.2ex,-0.45ex>*{},
(10.5,2)="C5"*{\vertex}+<-1.2ex,-0.45ex>*{},
(11.5,2)="C6"*{\vertex}+<-1.2ex,-0.45ex>*{},
(10.33,1)="C7"*{\vertex}+<-1.2ex,-0.45ex>*{}, 
(10.66,1)="C8"*{\vertex}+<-1.2ex,-0.45ex>*{},
(11.33,1)="C9"*{\vertex}+<-1.2ex,-0.45ex>*{}, 
(11.66,1)="C10"*{\vertex}+<-1.2ex,-0.45ex>*{},
(13,3)="C2"*{\vertex}+<1.2ex,-0.45ex>*{},
(12.5,2)="C11"*{\vertex}+<-1.2ex,-0.45ex>*{},
(13.5,2)="C12"*{\vertex}+<-1.2ex,-0.45ex>*{},
(12.33,1)="C13"*{\vertex}+<-1.2ex,-0.45ex>*{},
(12.66,1)="C14"*{\vertex}+<-1.2ex,-0.45ex>*{},
(13.33,1)="C15"*{\vertex}+<-1.2ex,-0.45ex>*{},  
(13.66,1)="C16"*{\vertex}+<-1.2ex,-0.45ex>*{},  
(12,4)="L0"*{\bullet}+<0ex,1.5ex>*{1},
"C2";"L0"**@{-},
"C1";"C5"**@{-},
"C1";"C6"**@{-},
"C5";"C7"**@{--},  
"C5";"C8"**@{--},
"L0";"C1"**@{-},
"C6";"C9"**@{--},
"C6";"C10"**@{--},
"C2";"C11"**@{-},
"C2";"C12"**@{-},
"C11";"C13"**@{--},
"C11";"C14"**@{--},
"C12";"C15"**@{--},
"C12";"C16"**@{--},
(0,3.5)="A0"*{\ }+<-1.2ex,-0.45ex>*{ \tau},
(1,3)="A1"*{\vertex}+<-1.2ex,-0.45ex>*{},
(0.5,2)="A5"*{\bullet}+<-1.2ex,-0.45ex>*{1},
(1.5,2)="A6"*{\vertex}+<-1.2ex,-0.45ex>*{},
(0.33,1)="A7"*{\vertex}+<-1.2ex,-0.45ex>*{},
(0.66,1)="A8"*{\vertex}+<-1.2ex,-0.45ex>*{},
(1.33,1)="A9"*{\vertex}+<-1.2ex,-0.45ex>*{},
(1.66,1)="A10"*{\vertex}+<-1.2ex,-0.45ex>*{},
(3,3)="A2"*{\vertex}+<1.2ex,-0.45ex>*{},
(2.5,2)="A11"*{\vertex}+<-1.2ex,-0.45ex>*{},
(3.5,2)="A12"*{\bullet}+<-1.2ex,-0.45ex>*{1},
(2.33,1)="A13"*{\vertex}+<-1.2ex,-0.45ex>*{},
(2.66,1)="A14"*{\vertex}+<-1.2ex,-0.45ex>*{},
(3.33,1)="A15"*{\vertex}+<-1.2ex,-0.45ex>*{},
(3.66,1)="A16"*{\vertex}+<-1.2ex,-0.45ex>*{},
(2,4)="A4"*{\vertex}+<0ex,1.5ex>*{},
"A2";"A4"**@{-},
"A1";"A5"**@{-},
"A1";"A6"**@{-},
"A5";"A7"**@{--},
"A5";"A8"**@{--},
"A4";"A1"**@{-},
"A6";"A9"**@{--},
"A6";"A10"**@{--},
"A2";"A11"**@{-},
"A2";"A12"**@{-},
"A11";"A13"**@{--},
"A11";"A14"**@{--},
"A12";"A15"**@{--},
"A12";"A16"**@{--},
(5,3.5)="B0"*{\ }+<-1.2ex,-0.45ex>*{ \beta_1 },
(6,3)="B1"*{\bullet}+<-1.2ex,-0.45ex>*{1}, 
(5.5,2)="B5"*{\vertex}+<-1.2ex,-0.45ex>*{},
(6.5,2)="B6"*{\vertex}+<-1.2ex,-0.45ex>*{},
(5.33,1)="B7"*{\vertex}+<-1.2ex,-0.45ex>*{},
(5.66,1)="B8"*{\vertex}+<-1.2ex,-0.45ex>*{},
(6.33,1)="B9"*{\vertex}+<-1.2ex,-0.45ex>*{},
(6.66,1)="B10"*{\vertex}+<-1.2ex,-0.45ex>*{},
(8,3)="B2"*{\vertex}+<1.2ex,-0.45ex>*{},
(7.5,2)="B11"*{\vertex}+<-1.2ex,-0.45ex>*{},
(8.5,2)="B12"*{\vertex}+<-1.2ex,-0.45ex>*{}, 
(7.33,1)="B13"*{\vertex}+<-1.2ex,-0.45ex>*{},
(7.66,1)="B14"*{\vertex}+<-1.2ex,-0.45ex>*{},
(8.33,1)="B15"*{\vertex}+<-1.2ex,-0.45ex>*{},
(8.66,1)="B16"*{\vertex}+<-1.2ex,-0.45ex>*{},
(7,4)="B4"*{\vertex}+<0ex,1.5ex>*{},
"B2";"B4"**@{-},
"B1";"B5"**@{-},
"B1";"B6"**@{-},
"B5";"B7"**@{--},
"B5";"B8"**@{--},
"B4";"B1"**@{-},
"B6";"B9"**@{--},
"B6";"B10"**@{--},
"B2";"B11"**@{-},
"B2";"B12"**@{-},
"B11";"B13"**@{--},
"B11";"B14"**@{--},
"B12";"B15"**@{--},
"B12";"B16"**@{--},
\endxy

 Consequently, in such way we construct second generating system for $A_{2^k}$ of $k$ elements that is less than in \cite{Iv}, and this system is minimal.


Let us consider function of Morse \cite{Shar} $f:\,{D^2}\to \mathbb{R}$ that painted at pict. 2 and graph of Kronrod-Reeb \cite{Maks} which obtained by contraction every set's component  of level of ${{f}^{-1}}(c)$ in point. Group of automorphism of this graph is isomorphic to $Sy{{l}_{2}}{{S}_{{{2}^{k}}}}$ where $k=2$ in general case we have regular binary rooted tree for arbitrary $k\in \mathbb{N}$.


\begin{figure}[h]
\begin{minipage}[h]{0.49\linewidth}
\center{\includegraphics[width=0.9\linewidth]{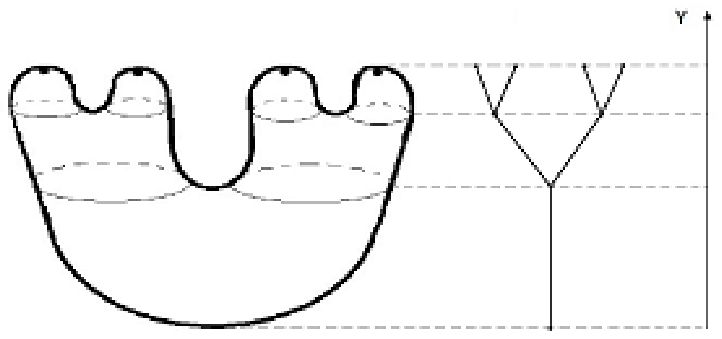} \\ Picture 2.}
\end{minipage}
\end{figure}
According to investigations of \cite{Maks2} for $D^2$ we have that $Syl_{2}S_{2^k} > G_k \simeq Syl_{2} A_{2^k}$ is quotient group of group of diffeomorphisms which stabilize function and isotopic to identity. Analogously to investigations of \cite{Maks, Maks2} there is short exact sequence
$0\to {{\mathbb{Z}}^{m}}\to {{\pi }_{1}}{{O}_{f}}(f)\to G\to 0$, where $G$-group of automorphisms Reeb's (Kronrod-Reeb) graph\cite{Maks} and $O_f(f)$ is orbit under action of group of diffeomorphisms, so it could be way to transfer it for a group $Sy{{l}_{2}}(S_{2^k})$, where $m$ in ${{\mathbb{Z}}^{m}}$ is number of inner edges (vertices) in Reeb's graph, in case for $Syl_{2}S_{4}$ we have $m=3$.

Higher half of projection of manifold from pic. 2 can be deteremed by product of the quadratic forms $-({{(x+4)}^{2}}+{{y}^{2}})({{(x+3)}^{2}}+{{y}^{2}})({{(x-3)}^{2}}+{{y}^{2}})({{(x-4)}^{2}}+{{y}^{2}})=z$ in points $(-4,0) (-3,0) (3,0) (4,0) $ it reach a maximum value 0. Generally there is $-d_{1}^{2}d_{2}^{2}d_{3}^{2}d_{4}^{2}=z$.

Also from Statement \ref{comm} and corollary from it about $(AutX^{[k]})'$ can be deduced that derived length of $Syl_2 A_2^k$ is not always equal to $k$ as it was said in Lemma 3 of \cite{Dm} because in case $A_{2^k}$ if $k=2$ its $Syl_2 A_4 \simeq K_4$ but $K_4$ is abelian group so its derived length is 1.

\section{ Conclusion }
The proof of minimality of constructed generating set was done, also the description of the structure $Syl_2 A_{2^n}$ and its property was founded.

\end{document}